\documentclass[letterpaper,11 pt]{article}

\usepackage{amsmath, amssymb, bbm, xspace}

\def\spose#1{\hbox to 0pt{#1\hss}}

\def\text #1{\hbox{\quad#1\quad}}

\def\nthinsp{\mskip -2   mu}

\def\superstar{^{\raise 0.5pt\hbox{$\nthinsp *$}}}
\def\SUPERSTAR{^{\raise 0.5pt\hbox{$*$}}}

\def\lamstarT {\lambda^{\raise 0.5pt\hbox{$\nthinsp *$}T}}

\def\hbar{\skew{4.2}\bar h}

		\def\bkE{{\rm I\kern-.17em E}}
		\def\bk1{{\rm 1\kern-.17em l}}
		\def\bkD{{\rm I\kern-.17em D}}
		\def\bkR{{\rm I\kern-.17em R}}
		\def\bkP{{\rm I\kern-.17em P}}
		\def\bkY{{\bf \kern-.17em Y}}
		\def\bkZ{{\bf \kern-.17em Z}}

		\def\beq{\begin{eqnarray}}
		\def\bc{\begin{center}}
		\def\be{\begin{enumerate}}
		\def\bi{\begin{itemize}}
		\def\bs{\begin{small}}
		\def\bS{\begin{slide}}
		\def\ec{\end{center}}
		\def\ee{\end{enumerate}}
		\def\ei{\end{itemize}}
		\def\es{\end{small}}
		\def\eS{\end{slide}}
		\def\eeq{\end{eqnarray}}

	\def\cp2problem#1#2#3#4{\fbox
		 {\begin{tabular*}{0.9\textwidth}
			{@{}l@{\extracolsep{\fill}}l@{\extracolsep{6pt}}l@{\extracolsep{\fill}}c@{}}
				#1 & & $#4 $ 
			\end{tabular*}}}

		\renewcommand{\emph}[1]{\textbf{#1}}

		\def\bkE{{\rm I\kern-.17em E}}
		\def\bk1{{\rm 1\kern-.17em l}}
		\def\bkD{{\rm I\kern-.17em D}}
		\def\bkR{\mathbb{R}}
		\def\bkP{{\rm I\kern-.17em P}}
		
		\def\bkZ{{\bf{Z}}}

\newcommand {\beeq}[1]{\begin{equation}\label{#1}}
\newcommand {\eeeq}{\end{equation}}
\newcommand {\bea}{\begin{eqnarray}}
\newcommand {\eea}{\end{eqnarray}}

\def\texitem#1{\par\smallskip\noindent\hangindent 25pt
               \hbox to 25pt {\hss #1 ~}\ignorespaces}

\usepackage{graphicx}
\usepackage{amsfonts,amsmath,fullpage,bbm}
\usepackage{amssymb,amsthm,multirow,verbatim}
\usepackage{acronym,wrapfig,plain,mathrsfs,enumerate,relsize,color}
\newtheorem{algorithm}{Algorithm}
\usepackage{algorithm}
\newtheorem{theorem}{Theorem}
\newtheorem{remark}{Remark}
\newtheorem{corollary}{Corollary}
\newtheorem{lemma}{Lemma}

\newtheorem{proposition}{Proposition}
\newtheorem{assumption}{Assumption}

\usepackage{subfig}
\usepackage{algorithmic}
\usepackage{pifont}
\usepackage{footmisc}
\usepackage{stackengine}
\newcommand\barbelow[1]{\stackunder[1.2pt]{$#1$}{\rule{1.5ex}{.075ex}}}
\usepackage[pdftex,colorlinks=true,urlcolor=blue,citecolor=black,anchorcolor=black,linkcolor=black]{hyperref}

\usepackage{algorithmic}

\newcommand{\aj}[1]{{\color{black}#1}}

\newcommand{\bmo}[1]{{\color{black}#1}}
\newcommand{\ze}[1]{{\color{black}#1}}
\newcommand{\zal}[1]{{\color{black}#1}}

\begin{document}
\allowdisplaybreaks
\title{An Inexact Variance-Reduced Method For Stochastic Quasi-Variational
Inequality Problems With An Application In Healthcare}

\author{Zeinab Alizadeh \footnote{ Systems and Industrial Engineering, University of Arizona,
 Tucson, AZ 85721, USA.} \and Brianna M. Otero \footnotemark[1] \and Afrooz Jalilzadeh \footnotemark[1]
}

\date{}

\maketitle

\begin{abstract}
This paper is focused on a stochastic quasi-variational inequality (SQVI) problem with a continuous and strongly-monotone mapping over a closed and convex set where the projection onto the constraint set may not be easy to compute. We present an inexact variance reduced stochastic scheme to solve SQVI problems and analyzed its convergence rate and oracle complexity. A linear rate of convergence is obtained by progressively increasing sample-size and approximating the projection operator. Moreover, we show how a competition among blood donation organizations can be modeled as an SQVI and we provide some preliminary simulation results to validate our findings.
\end{abstract}
\section{INTRODUCTION}
\label{sec:intro}
Variational inequality (VI) problems have a broad range of applications in convex Nash games, traffic equilibrium problems, economic equilibrium problems, amongst others \cite{facchinei2007finite}. Stochastic VI (SVI) has been proposed in order to describe decision making problems which involve uncertainty. Such an \zal{uncertainty} commonly arises in simulation optimization and stochastic economic equilibrium involving expectations \cite{gurkan1996sample}. In this paper, we study a stochastic quasi-VI (SQVI) problem which is an extension of SVI when the convex sets where the solutions are to be found depend on the solutions themselves. Let $X$ be a finite-dimensional real vector space. Consider the following SQVI problem: find $x\in K(x)$ such that
\begin{align}\label{sqvi}
\langle F(x), y-x \rangle\geq 0, \quad \forall y\in K(x),
\end{align}
where $K:X\to 2^X$ is a set-valued mapping with non-empty bounded closed convex values $K(x)\subseteq X$ for all $x\in X$, $F(x)\triangleq \mathbb E[G(x,\xi)]$, $\xi: \Omega \to \mathbb R^d$, ${G}: X \times \mathbb R^d  \rightarrow
\mathbb{R}^n$, and the associated probability space is denoted by $(\Omega, {\cal F}, \mathbb{P})$. If $K(x)=K$, then problem \eqref{sqvi} turns into a conventional SVI problem. 

While deterministic VIs \cite{malitsky2015projected} and SVIs \cite{jalilzadeh2019proximal} have received significant study over the last several decades, less is known regarding SQVIs. In the deterministic regime, there have been several studies about numerical methods to solve QVIs \cite{pang2005quasi,antipin2013second,facchinei2014solving,mijajlovic2015proximal,noor2000new,noor2007existence,ryazantseva2007first,salahuddin2004projection}. Recently, a linear convergence rate for strongly monotone QVI problem has been obtained by \cite{mijajlovic2019gradient} (see also \cite{nesterov2006solving}). However, there are no available rate results for SQVI problems to the best of our knowledge. Moreover, in many scenarios, computing the projection onto the constraint set may be expensive or may not have an analytic solution. In this work, we propose a variance-reduced stochastic scheme to solve problem \eqref{sqvi} with convergence guarantee by computing the projection inexactly at each iteration. Next, to show the need to model problems as SQVI, instead of VI, we illustrate a real-world problem arising in healthcare in which projecting onto the constraint might also be challenging and we discuss the existing gaps and the main contributions of this paper. 

\subsection{Applications and Existing Gaps}\label{sec:blood} 
Recently QVI problems have emerged in many application domains including communication networks, wireless sensor networks, power control \cite{stupia2015power,tang2018securing} and healthcare \cite{nagurney2017supply,nagurney2019competition}. To further motivate our
research goals, we discuss one problem that \bmo{arises} in healthcare in details which can be formulated
as an SQVI problem.

{\bf Blood Donation Problem.} Inspired by \cite{nagurney2019competition},  {we consider the competition for blood donations among blood service organizations}, where each organization intends to maximize their transaction utilities and compete on the quality of service that they provide in different regions. Suppose we have $n$ organizations providing service in $m$ different locations. Each organization has the quality of service as its strategic variable. We denote the level of service quality of organization $i$ in location $j$ by $Q_{ij}$ and we group the level of service quality for all blood service organizations into matrix $Q\in \mathbb R^{n\times m}$. We assume there is an upper bound and lower bound for the quality level that each organization can provide, and $K_i$ denotes the feasible set of organization $i$, hence we denote the feasible set of all players in the game by $K=\prod_{i=1}^n K_i$.  
Each organization seeks to maximize its transaction utilities, $U_i$, and Nash equilibrium is established if no blood service organization can improve upon its transaction utility by altering its quality service levels, given that the other organizations have decided on their quality service levels.
The associated VI formulation for this Nash Equilibrium problem can be characterized by finding a quality service level pattern $Q^*\in K$ such that the following holds: 
\begin{align}\label{vi1}-\sum_{i=1}^n\sum_{j=1}^m \tfrac{\partial U_i(Q^*)}{\partial Q_{ij}}\times (Q_{ij}-Q_{ij}^*)\geq 0, \quad \forall Q\in K.\end{align}

{\bf Gaps.} There are two main gaps in the above formula presented by \cite{nagurney2019competition}.
\begin{itemize}
\item[(i)] The total volume of blood donations in one location should be bounded from below to make sure we meet the demand on that location, i.e., $\sum_i{P_{ij}(Q)\geq P_j}$.
Therefore, the feasible set of organization $i$ will depend on {the strategy of other players} and problem \eqref{vi1} will change to a QVI problem. Moreover, depending on  $P_{ij}(\cdot)$, the projection onto such a constraint set may not be easy to compute.
\item[(ii)] The stochasticity of the parameters is ignored. For instance, the cost of collecting blood is an uncertain parameter, hence, the transaction utility $U_i$ is stochastic.
\end{itemize}

As continually emphasized in the literature \cite{nagurney2019competition,nagurney2017supply}, the main reason why these types of competitions are always formulated simply as VI and the dependency of player $i$'s strategy on other players' strategies is ignored, is that solving a QVI/SQVI problem is more complicated and challenging than solving the VI/SVI counterpart. There is no efficient method for solving such problems. 

\subsection{Contributions} To fill the aforementioned gaps, in this paper, we consider strongly monotone SQVI problems. We develop and analyze an inexact variance-reduced stochastic scheme (Inexact-VR-SQVI). In particular, we investigate the convergence rate of the proposed method under the conditions where the projection onto the constraint may or may not be easy to compute. In the latter scenario, at each iteration of the proposed method, the projection step is solved inexactly and the effect of the underlying error on the convergence rate is characterized. More importantly, by improving the accuracy of such approximation at an appropriate rate combined with a variance reduction technique, we demonstrate a linear convergence rate that matches the best-known rate result in the deterministic counterpart. We also show that achieving an $\epsilon$-solution, i.e., $\mathbb E[\|x_k-x^*\|]\leq \epsilon$, requires $\mathcal O(1/\epsilon^2)$ sample operators. To the best of our knowledge, this is the first convergence rate result for SQVI problems.

Next, we state the main assumptions that are needed for the convergence analysis. In Section \ref{sec:conv}, we introduce the Inexact-VR-SQVI algorithm and in Section \ref{sec:numerical}, we show the performance of the proposed scheme by implementing it on the blood donation problem that is modeled as SQVI. Finally, Section \ref{sec:conclude} presents our main conclusions and future work.

\subsection{Assumptions}
In this section, first we define important notations and then the main assumptions that we need for the convergence analysis are stated.

{\bf Notations.}  Throughout the paper, $\|x\|$ denotes the Euclidean vector norm, i.e., $\|x\|=\sqrt{x^Tx}$.  $\mathbf{P}_X [u]$ is the projection of $u$ onto the set $X$, i.e. $\mathbf{P}_X [u] = argmin_{z \in X} \| z-u \|$. $\mathbb E[x]$ is used to denote the expectation of a random variable $x$.
\begin{assumption}\label{asump1}
Assume that operator $F:X\rightarrow \mathbb R^n$ is $\mu$-strongly monotone
\begin{align*}
\langle F(x)-F(y),x-y\rangle\geq \mu \|x-y\|^2, \quad \forall x,y\in X,
\end{align*} 
and $L$-Lipschitz continuous on $X$
\begin{align*}
\|F(x)-F(y)\|\leq L\|x-y\|,\quad \forall x, y \in X.
\end{align*}

\end{assumption}
If $\mathcal F_{k}$ denotes the information history at epoch $k$, then we have the following requirements on the associated filtrations where $\bar w_{k,N_k} \triangleq \tfrac{1}{N_k}{\sum_{j=1}^{N_k} \zal{( G(x_k,\xi_{j,k})-F(x_k))}}$.
\begin{assumption}\label{assump_error}
 There exists $\nu>0$ such that $\mathbb E[\bar w_{k,N_{k}}\mid  \mathcal F_{k}]=0$ and $\mathbb E[\| \bar w_{k,N_{k}}\|^2\mid  \mathcal F_{k}]\leq \tfrac{\nu^2}{N_{k}}$  holds almost surely 
 for all $k$, 
where $\mathcal{F}_k \triangleq \sigma\{x_0, x_1, \hdots, x_{k-1}\}$.  
\end{assumption} 
In our analysis, it is assumed that an inexact solution of the projection operator exists through an inner algorithm $\mathcal A$ satisfying the following assumption. Later, in section \ref{sec:inner}, instances of algorithms satisfying this assumption {are discussed}. 
\begin{assumption}\label{assump:inner}
There is an iterative method $\mathcal A$ that satisfies the following property:
For any $x\in\mathbb R^n$, any closed and convex set $K\subseteq \mathbb R^n$, and an initial point $u_0$, $\mathcal A$ can generate an output $u\in \mathbb R^n$ such that $\|u-\tilde u\|^2\leq C/t^2$ for some $C>0$ satisfying $\tilde u=\mbox{argmin}_{y\in K}\{\tfrac{1}{2}\|y-x\|^2\}$. 
\end{assumption}
\section{CONVERGENCE ANALYSIS}\label{sec:conv}
In our analysis, the following technical lemma for projection mappings {is used}.
\begin{lemma}\label{lem1}\cite{bertsekas2003convex}
\noindent Let $ X\subseteq \mathbb{R}^n $  be a nonempty closed and convex set. Then the following hold:
(a) $\|\mathbf{P}_X [u]- \mathbf{P}_X [v]\| \leq \|u-v\| $ for all $ u,v \in \mathbb{R}^n$;
(b) $ (\mathbf{P}_X [u]-u)^T(x-\mathbf{P} _X [u]) \geq 0 $ for all $u \in \mathbb{R}^n$ and $x \in X$.  
\end{lemma}
The main difference between VIs and QVIs is in the existence of the solution. It is well-known that if operator $F$ is strongly monotone and Lipschitz continuous on a closed and convex set, then corresponding VI (and SVI) has a unique solution \cite{nesterov2006solving,jalilzadeh2019proximal}. However, these conditions are not sufficient for the existence of the QVI solutions. In the following proposition, we state the requirements needed in our analysis in order to prove the existence of a solution for QVIs (and similarly for SQVIs). 

\begin{proposition}\label{t1}\cite{noor1994general}
Suppose Assumption \ref{asump1} holds and there exists $\gamma>0$ such that $\|\mathbf{P}_{K(x)} [u]- \mathbf{P}_{K(y)} [u]\| \leq \gamma \|x-y\| $ for all $ x,y,u \in X$ and $\gamma+\sqrt{1-\mu^2/L^2}<1$. Then, problem \eqref{sqvi} has a unique solution. 
\end{proposition}

More discussion on the existence of a solution to an SQVI problem can be found in \cite{ravat2017existence}. In Algorithm \ref{alg1}, a variance-reduced stochastic scheme for solving SQVI problem \eqref{sqvi} {is presented}. In particular, at each iteration, a step along the negative direction of the sample-average operator $G(\cdot,\xi)$, with step size $\eta$ is taken following by computing an inexact solution, $y_k$, of the projection onto the set $K(x_k)$ using Algorithm $\mathcal A$. The next iterate point is calculated based on a carefully selected convex combination of the previous iterates, $x_k$, and $y_k$.
In our analysis, $e_k$ denotes the error of computing the projection operator, i.e., for any $k\geq 0$
$    e_k\triangleq u_k-\mathbf{P}_{K(x_k)}\left[x_k-\eta\frac{\sum_{j=1}^{N_k}G(x_k,\xi_{j,k})}{N_k}\right].$
In Theorem \ref{rate}, we derive the expected solution error bound in terms of $e_k$. Then, in Corollary \ref{rate:main}, we obtain the rate and complexity statements for the Algorithm \ref{alg1}.  
\begin{algorithm}[htbp]
\caption{Inexact-VR-SQVI}
\label{alg1}
{\bf Input}: $x_0\in X$, $\eta>0$,  $\{N_k\}_k$, $\{t_k\}_k$, $\{\alpha_k\}_k$ and Algorithm $\mathcal A$ satisfying Assumption \ref{assump:inner}; \\
{\bf for $k=0,\hdots T-1$ do}\\
\mbox{(1)} Find an approximate solution of the following projection using Algorithm $\mathcal A$ \zal{ in $t_k$ iterations}
$$y_k\approx \mathbf{P}_{K(x_k)}\left[x_k-\eta\frac{\sum_{j=1}^{N_k}G(x_k,\xi_{j,k})}{N_k}\right];$$ 
\mbox{(2)} $x_{k+1}=(1-\alpha_k)x_k+\alpha_k y_k$;  \\
{\bf end for}\\
{\bf Output:} $x_{k+1}$;
\end{algorithm}

\begin{theorem}\label{rate}
Consider the iterates generated by Algorithm \ref{alg1} and suppose Assumptions \ref{asump1} and \ref{assump_error} hold. Choose $\alpha_k=\bar\alpha\in(0,1)$ and define $\beta\triangleq\gamma+\sqrt{1+L^2\eta^2-2\eta\mu}$, $q\triangleq (1-\beta)\bar\alpha$. Choose stepsize $\eta$ such that the following holds:
$$|\eta-\tfrac{\mu}{L^2}|<\tfrac{{\sqrt {\mu^2-L^2(2\gamma-\gamma^2)}}}{L^2}.$$
Let $N_k=\lceil \rho^{-2k}\rceil$ for all $k>0$ where $\rho>1-q$. Then the following holds:
\begin{align}\label{eq:rate-general}
\mathbb E\left[ \| x_T-x^{*} \|\right]& 
\leq \rho^{T} \|x_{0}-x^{*}\| +\bar{\alpha} \eta \nu \rho^{T-1} + \frac{\bar{\alpha} \eta \nu \rho^{T}}{\rho+q-1} +{\bar\alpha\sum_{k=\ze{0}}^{T-1}\left(\|e_k\|\rho^{T-1-k}\right)}.
\end{align}
\end{theorem}
\begin{proof}
Recall that $\bar w_{k,N_k}= \zal{\tfrac{1}{N_k}{\sum_{j=1}^{N_k} ( G(x_k,\xi_{j,k})-F(x_k) )}} $. Using the update rule of $x_{k+1}$ in Algorithm \ref{alg1} and the fact that $e_k$ denotes the error of computing
the projection operator, we obtain the following.  
\begin{align}\label{b1}
\nonumber&\|x_{k+1}-x^*\|\\\nonumber \quad&=\left\|(1-\alpha_k)x_k+\alpha_k \mathbf{P}_{K(x_k)}\left[x_k-\eta(F(x_k)+\bar w_{k,N_k})\right]+{\alpha_ke_k}-(1-\alpha_k)x^*-\alpha_k\mathbf{P}_{K(x^*)}\left[x^*-\eta F(x^*)\right]\right\|\\
\nonumber&\leq \|(1-\alpha_k)(x_k-x^*)\|+\alpha_k\|\mathbf{P}_{K(x_k)}\left[x_k-\eta(F(x_k)+\bar w_{k,N_k})\right]-\mathbf{P}_{K(x^*)}\left[x_k-\eta(F(x_k)+\bar w_{k,N_k})\right]\|\\
\nonumber&\quad +\alpha_k\|\mathbf{P}_{K(x^*)}\left[x_k-\eta(F(x_k)+\bar w_{k,N_k})\right]-\mathbf{P}_{K(x^*)}\left[x^*-\eta F(x^*)\right]\|+{\alpha_k\|e_k\|}\\
&\leq \|(1-\alpha_k)(x_k-x^*)\|+\alpha_k\gamma \|x_k-x^*\|+\alpha_k\underbrace{\|x_k-x^*-\eta(F(x_k)-F(x^*))\|}_{\text{term (a)}}+\alpha_k\eta\|\bar w_{k,N_k}\|+\alpha_k\|e_k\|,
\end{align}
where in the last inequality we used Lemma \ref{lem1} and Proposition \ref{t1}. Now using strong monotonicity and Lipschitz continuity, we can bound term (a) in inequality \eqref{b1}. 
\begin{align}\label{b2}
\nonumber\|x_k-x^*-\eta(F(x_k)-F(x^*))\|^2&=\|x_k-x^*\|^2+\eta^2\|F(x_k)-F(x^*)\|^2-2\eta\langle x_k-x^*,F(x_k)-F(x^*)\rangle\\
\nonumber&\leq (1+L^2\eta^2-2\eta\mu)\|x_k-x^*\|^2\\& \implies \text{term(a)}\leq \sqrt{1+L^2\eta^2-2\eta\mu}\|x_k-x^*\|.
\end{align}
Using \eqref{b2} in \eqref{b1}, defining $\beta\triangleq\gamma+\sqrt{1+L^2\eta^2-2\eta\mu}$ and $q_i\triangleq (1-\beta)\alpha_i$ we get the following:
\begin{align}\label{b3}
\nonumber\|x_{k+1}-x^*\|&\leq (1-\alpha_k)\|x_k-x^*\|+\alpha_k\left(\gamma+\sqrt{1+L^2\eta^2-2\eta\mu}\right)\|x_k-x^*\|+\alpha_k\eta\|\bar w_{k,N_k}\| \zal{+\alpha_k\|e_k\|}\\
\nonumber&=(1-(1-\beta)\alpha_k)\|x_k-x^*\|+\alpha_k\eta\|\bar w_{k,N_k}\|+\alpha_k\|e_k\|\\
\nonumber&\leq  \prod_{i=0}^k (1-q_{i})\|x_{0}-x^{*} \| + \sum_{i=0}^{k-1}\left(\left(\prod_{j=i}^{k-1}(1-q_{j+1})\right) \alpha_{i} \left(\eta \|\bar{w}_{i,N_{i}}\|+\|e_i\|\right)\right)\\
& \quad + \alpha_{k} \left(\eta \|\bar{w}_{k,N_{k}}\|+\|e_k\|\right).
\end{align} 
For any $k$, we choose $\alpha_k=\bar \alpha$, where $0<\bar\alpha<1$. Based on the conditions of the theorem, one can easily verify that $\beta<1$ and  $q_k=q<1$ for all $k\geq 0$. Now, by choosing $N_k=\lceil \rho^{-2k}\rceil$, where {$\rho\geq1-q$}, it follows from inequality \eqref{b3} and Assumption \ref{assump_error} by taking expectation from both \bmo{sides} that for any $T\geq 1$, 
\begin{align*}
\nonumber\mathbb E\left[ \| x_{T}-x^{*} \|\right] &\leq (1-q)^{T} \|x_{0}-x^{*}\| + \bar{\alpha} \eta \sum_{k=0}^{T-2}\left((1-q)^{T-1-k} \tfrac{\nu}{\rho^{-k}} \right)\\&\quad+{\bar\alpha\sum_{k=0}^{T-1}\left(\|e_k\|(1-q)^{T-1-k}\right)} + \bar{\alpha} \eta (\nu/ \rho^{-T+1}).
\end{align*}
Using the fact that $\rho \geq 1-q$, the following holds.  
\begin{align*}
\mathbb E\left[ \| x_T-x^{*} \|\right]& \leq\rho^{T} \|x_{0}-x^{*}\| +\bar{\alpha} \eta \nu \rho^{T-1} + \bar{\alpha} \eta \nu \rho^\ze{T-1} \sum_{\ze{k}=0}^{T-2} ((1-q)/\rho)^{T-1-\ze{k}}+{\bar\alpha\sum_{k=0}^{\ze{T-1}}\left(\|e_k\|\rho^{T-1-k}\right)}\\&
\leq \rho^{T} \|x_{0}-x^{*}\| +\bar{\alpha} \eta \nu \rho^{T-1} + \frac{\bar{\alpha} \eta \nu \rho^{T}}{\rho+q-1} +{\bar\alpha\sum_{\aj{k=0}}^{T-1}\left(\|e_k\|\rho^{T-1-k}\right)},
  \end{align*}
 where in the last inequality we used the fact that $\sum_{k=0}^{T-2} ((1-q)/\rho)^{T-1-k}\leq \sum_{j=1}^{T-1} (\tfrac{1-q}{\rho})^j\leq \tfrac{\rho}{\rho+q-1}$. 
 \end{proof}

 \begin{corollary}\label{rate:main}
Under the premises of Theorem \ref{rate} and selecting $t_k=\tfrac{({k+1})\log^2(k+2)}{\rho^{k}}$, where $t_k$ is the number of steps for algorithm $\mathcal A$ at each iteration $k$, then,   

{\bf (i)}  the following holds:
\begin{align*}
\mathbb E[\|x_{T}-x^*\|]\leq  \rho^{T} \|x_{0}-x^{*}\| +\bar{\alpha} \eta \nu \rho^{T-1} + \frac{\bar{\alpha} \eta \nu \rho^{T}}{\rho+q-1}+\bar\alpha CD \rho^{T-1}=\mathcal O(\rho^{T}),
\end{align*}
where $D\triangleq\sum_{\aj{k=0}}^{\aj{\infty}} \tfrac{1}{(k+1)\log^2(k+2)}\aj{\leq 3.39}$.

{\bf (ii)} to compute a solution $x_{T}$ such that  $E[\|x_{T}-x^*\|]\leq \epsilon$, the total number of sample operators is  $\sum_{k=0}^{T-1} N_k\geq \mathcal O(1/\epsilon^2)$.
\end{corollary}
 \begin{proof}
{\bf (i)} Recall that $e_k$ represents the error of computing the projection at iteration $k$. 
According to the assumption \ref{assump:inner}, Algorithm $\mathcal A$ has a convergence rate of $C/t_k^2$ within $t_k$ inner steps. By selecting $t_k=\tfrac{({k+1})\log^2(k+2)}{\rho^{k}}$ we conclude that $\|{e_k}\|\leq \tfrac{C}{t_k}=\tfrac{C\rho^{k}}{(k+1)\log^2(k+2)}$. 
Therefore, the following holds: $$\bar\alpha\sum_{\aj{k=0}}^{T-1}\left(\|e_k\|\rho^{T-1-k}\right)\leq \bar\alpha C\rho^{T-1}\sum_{\aj{k=0}}^{T-1} \tfrac{1}{(k+1)\log^2(k+2)}\leq \bar\alpha CD\rho^{T-1},$$ 
 where we let $D=\sum_{\aj{k=0}}^{\aj{\infty}} \tfrac{1}{(k+1)\log^2(k+2)}\aj{\leq 3.39}$. Therefore, we obtain
\begin{align}\label{rate_proof}
\mathbb E[\|x_{T}-x^*\|]\leq  \rho^{T} \|x_{0}-x^{*}\| +\bar{\alpha} \eta \nu \rho^{T-1} + \frac{\bar{\alpha} \eta \nu \rho^{T}}{\rho+q-1}+\bar\alpha CD \rho^{T-1}=\mathcal O(\rho^{T}).
\end{align}
{\bf (ii)} To compute an $\epsilon$-solution, i.e., $E[\|x_{T}-x^*\|]\leq \epsilon$, it follows from \eqref{rate_proof} that $T\geq \log_{1/\rho}(\bar D/\epsilon)$ iterations is required, where $\bar D=\|x_0-x^*\|+\bar\alpha\eta\nu\rho^{-1}+\tfrac{\bar\alpha\eta\nu}{\rho+q-1}+\bar\alpha CD\rho^{-1}$. Hence, we obtain $$\sum_{k=0}^{T-1} N_k\geq \frac{\rho^2}{1-\rho^2}\left(\frac{\bar D^2}{\epsilon^2}-1\right).$$ 
\end{proof}
\begin{remark}[{\bf Total number of inner iterations}]
In Algorithm \ref{alg1}, each iteration requires taking $t_k=\tfrac{({k+1})\log^2(k+2)}{\rho^{k}}$ inner steps of Algorithm $\mathcal A$. Therefore, the total number of inner iterations is
$$\sum_{k=0}^{T-1} t_k=\sum_{k=0}^{T-1} \tfrac{({k+1})\log^2(k+2)}{\rho^{k}}\leq {T}\log^2(T+1)\tfrac{(1/\rho)^{T}}{1/\rho-1}.$$
To achieve an $\epsilon$-solution, we have $T=\log_{1/\rho}{\bar D/\epsilon}$, hence one can obtain 
$\sum_{k=0}^{T-1}t_k\leq \mathcal O({1\over \epsilon}\log(1/\epsilon))$.
\end{remark}
\begin{remark}[{\bf Exact-VR-SQVI}]
The solution error bound obtained in \eqref{eq:rate-general} represents a general convergence rate in terms of the error of the projection operator. The decay of this error governs the convergence rate of the algorithm. In particular, in Corollary \ref{rate:main} we characterized the rate of decay of this error to guarantee a linear convergence rate. In other extreme, 
when the projection onto the constraint set is easy to compute, i.e., $\|e_k\|=0$ for all $k\geq0$. Then under the premises of Theorem \ref{rate}, a linear convergence rate for Algorithm \ref{alg1} can be obtained. In particular, the following bound for the expected solution error holds:
\begin{align*}\mathbb E\left[ \| x_T-x^{*} \|\right]& 
\leq \rho^{T} \|x_{0}-x^{*}\| +\bar{\alpha} \eta \nu \rho^{T-1} + \frac{\bar{\alpha} \eta \nu \rho^{T}}{\rho+q-1}.
  \end{align*}
\end{remark}

\subsection{Instances of the Inner Algorithm $\mathcal A$}\label{sec:inner}
As discussed in section \ref{sec:conv}, when the projection onto the constraint set $K(x)$ is not easy to compute, one needs to use an approximation of such operator. Indeed, such an approximation can be obtained via implementing Algorithm $\mathcal A$ with a progressive accuracy at each iteration. 

Here we consider a general class of convex constraint set comprises of convex functional constraints. In particular, we assume that $K(x)=\{y\in X\mid g_i(x,y)\leq 0,~i\in\{1,\hdots,m\}\}$, where $g_i(x,\cdot):X\to \mathbb R$ is convex for any $x\in X$ and $i\in\{1,\hdots,m\}$. Therefore, at each iteration of Algorithm \ref{alg1} one needs to compute the projection operator inexactly which is of the following form:
\begin{align}\label{co}
&\min_{u\in K(x)} \quad {1\over 2}\left\|u- x\right\|^2, 
\end{align}
for some given $x\in\mathbb{R}^n$. Problem \eqref{co} has a strongly convex objective function with nonlinear convex constraints, and there has been a variety of methods developed in the optimization literature to solve such a problem. One of the efficient class of methods for solving large-scale convex constrained optimization problem with strongly convex objective  {that satisfies Assumption \ref{assump:inner}} is the first-order primal-dual scheme guaranteeing a convergence rate of $\mathcal O(1/t^2)$, where $t$ denotes the number of iterations, in terms of suboptimality and infeasibility, e.g., \cite{he2015mirror,malitsky2018proximal} and \cite{hamedani2018primal}. 

For instance, Accelerated Primal-Dual with Backtracking (APDB) method introduced by \cite{hamedani2018primal} with an initial point $ u_0$ and output $u$ has a convergence rate of $\mathcal O(1/t^2)$ within $t$ steps, where $\|u-\tilde u\|^2\leq (a_1\| u_0-\tilde u\|^2+a_2)/t^2$ for some $a_1,a_2>0$, hence, satisfying the condition of Assumption \ref{assump:inner}. In the next section, we use APDB as an instance of Algorithm $\mathcal A$ for solving the projection operator inexactly for various numerical experiments.  

\section{NUMERICAL EXPERIMENTS}\label{sec:numerical}

In this section, we consider two blood donation examples inspired, in part, by the American Red Cross (cf. \cite{webs1}). We concentrate on Tucson, Arizona, where the American Red Cross and United Blood Services compete with each other. {The experiments are performed on Matlab (2021) on a 64-bit Windows 11 with Intel i5-1135G7 @2.4GHz with 8GB RAM. 
Inspired by the problems considered in \cite{nagurney2019competition}, we consider two settings for the blood donation problem. In each example, we implemented two variants of the VR-SQVI algorithm; exact-VR-SQVI and inexact-VR-SQVI. For inexact-VR-SQVI, we used APDB \cite{hamedani2018primal} as an instance of Algorithm $\mathcal A$. In exact-VR-SQVI, to solve projection subproblem, we use commercial optimization solver MOSEK through CVX \cite{grant2014cvx}. We then demonstrate the advantage of inexact approach in terms of the running time of the algorithm when constraint set is not easy to project.} 

   \begin{wrapfigure}{r}{0.23\textwidth}
	\vspace{-0.1in}
	\hspace{-0.5in}
  \begin{center}
\includegraphics[width = 3.5 cm]{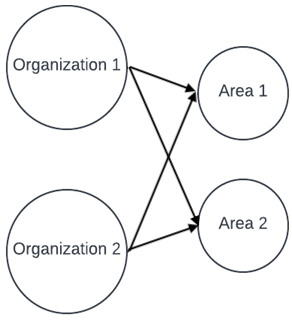}
    \caption{The network structure.}
    \label{f2}
    \end{center}
\end{wrapfigure}

  \textbf{Example 1.} { In this example, we consider two blood service organizations: the American Red Cross and United Blood Services. These organizations correspond to organizations 1 and 2 in Figure \ref{f2}. Formed in 1943, United Blood Services is a nonprofit company based in Arizona that offers blood and services to over 500 hospitals in 18 states. In this example, both United Blood Services and the American Red Cross have stationary areas to donate blood in Tucson. These are represented by the two area nodes in Figure \ref{f2}. 
  
  Consider the blood donation problem defined in Section \ref{sec:blood} where the utility associated with the blood service organization $i$ is denoted by $\omega_i\sum_{j=1}^m\gamma_{ij} Q_{ij}$, where $\omega_i$ and $\gamma_{ij}$ are positive numbers. The cost associated with collecting blood in location $j$ by organization $i$ is denoted by $c_{ij}(Q)$. Moreover, $P_{ij}(Q)$ represents a volume of blood donations in location $j$ by organization $i$ and we associate an average price $\pi_i$ for blood service organization $i$. 
Therefore, the transaction utility $U_i$ can be defined as  $$U_i\triangleq \pi_i\sum_{j=1}^m P_{ij}(Q)+\omega_i\sum_{j=1}^m\gamma_{ij} Q_{ij}-\sum_{j=1}^m c_{ij}(Q).$$}
  

The American Red Cross has a baseline of 130 and 135 repeat donors. United Blood Services has lower baseline populations of 123 and 135. These monthly values are represented in the following four transaction utility functions. 
The volume of blood donations in both locations for the American Red Cross are:
\begin{align*}
& P_{11} (Q ) = 10Q_{11} - Q_{21} - Q_{22} + 130 
\\&P_{12} (Q ) = 12Q_{12} - Q_{21} - 2Q_{22} + 135 .
\end{align*}
 
The volume of blood donations in both locations for the United Blood Services
are:
 \begin{align*}
&P_{21} (Q ) = 11Q_{21} - Q_{11} - Q_{12} + 123 
\\&P_{22} (Q ) = 12Q_{22} - Q_{11} - Q_{12} + 135 .
 \end{align*}
 

The utility function components of the transaction utilities of
these blood service organizations are:
\begin{align*}
    &\omega_{1} = 9, \quad \gamma_{11} = 8, \quad \gamma_{12} = 9.\\
    &\omega_{2} = 10,\quad \gamma_{21} = 9, \quad \gamma_{22} = 10.
\end{align*}

{In this example, blood collection sites must pay for employees, supplies, energy, and providing the level of quality service. The uncertainty of the total mentioned operating costs over time are represented in the following functions:}    
\begin{align*}
{c_{11} (Q,\xi ) = (5+\xi)Q_{11}^2 + 10000},\quad c_{12} (Q,\xi ) = (18+\xi)Q_{12}^2 + 12000,\\
c_{21} (Q,\xi ) = (4.5+\xi)Q_{21}^2 + 12000,\quad c_{22} (Q,\xi ) = (5+\xi)Q_{22}^2 + 14000,
\end{align*}
where $\xi$ has an i.i.d. standard normal distribution.
The lower and upper bounds on the quality levels are considered as follows:
\begin{align*}
\barbelow Q_{11} = 50,  \overline Q_{11} = 80,\quad \barbelow Q_{12} = 40,  \overline Q_{12} = 70,\\
\barbelow Q_{21} = 60, \overline Q_{21} = 90,\quad \barbelow Q_{22} = 70,  \overline Q_{22} = 90.
\end{align*}
The prices, which correspond to the collection component of
the blood supply chain, are
$\pi_1 = 70,\quad \pi_2 = 60.$
We consider the minimum volume of blood required in location 1  and 2 are $P_{1} = 1200$ and  $P_{2} = 1100$, respectively. Hence, the blood volume requirement is depicted by the following constraints:
\begin{align*}
& 9 Q_{11} + 10 Q_{21} - Q_{22} - Q_{12} + 253 \geq 1200, \\
& 11Q_{12} - Q_{21} +10 Q_{22} - Q_{11} + 270 \geq 1100. 
\end{align*}

{By implementing Inexact-VR-SQVI algorithm, we obtain the following solutions}
\begin{align*}
    Q_{11}^* = 72.81 ,\quad Q_{12}^* = 40.00,\quad Q_{21}^* = 78.09,\quad Q_{22}^* = 77.59,  
\end{align*}
and optimal value for each organization with the following values: $U_1(Q^*) = 7065$ and  $U_2(Q^*) = 40589 $. 

Figure \ref{plotex1} illustrates the progress of Inexact-VR-SQVI in terms of the relative suboptimality for the utility function of each organization versus the running time. 
From Table \ref{Ex1Table}, one can observe that Exact-VR-SQVI algorithm takes 97 seconds for a simple 2-dimensional example while the Inexact-VR-SQVI algorithm can obtain an approximated solution in 0.5 \bmo{seconds} with \bmo{a} relative accuracy of less than $10^{-6}$.
 \begin{figure}[htb]
    \centering
    \includegraphics[width = 10 cm]{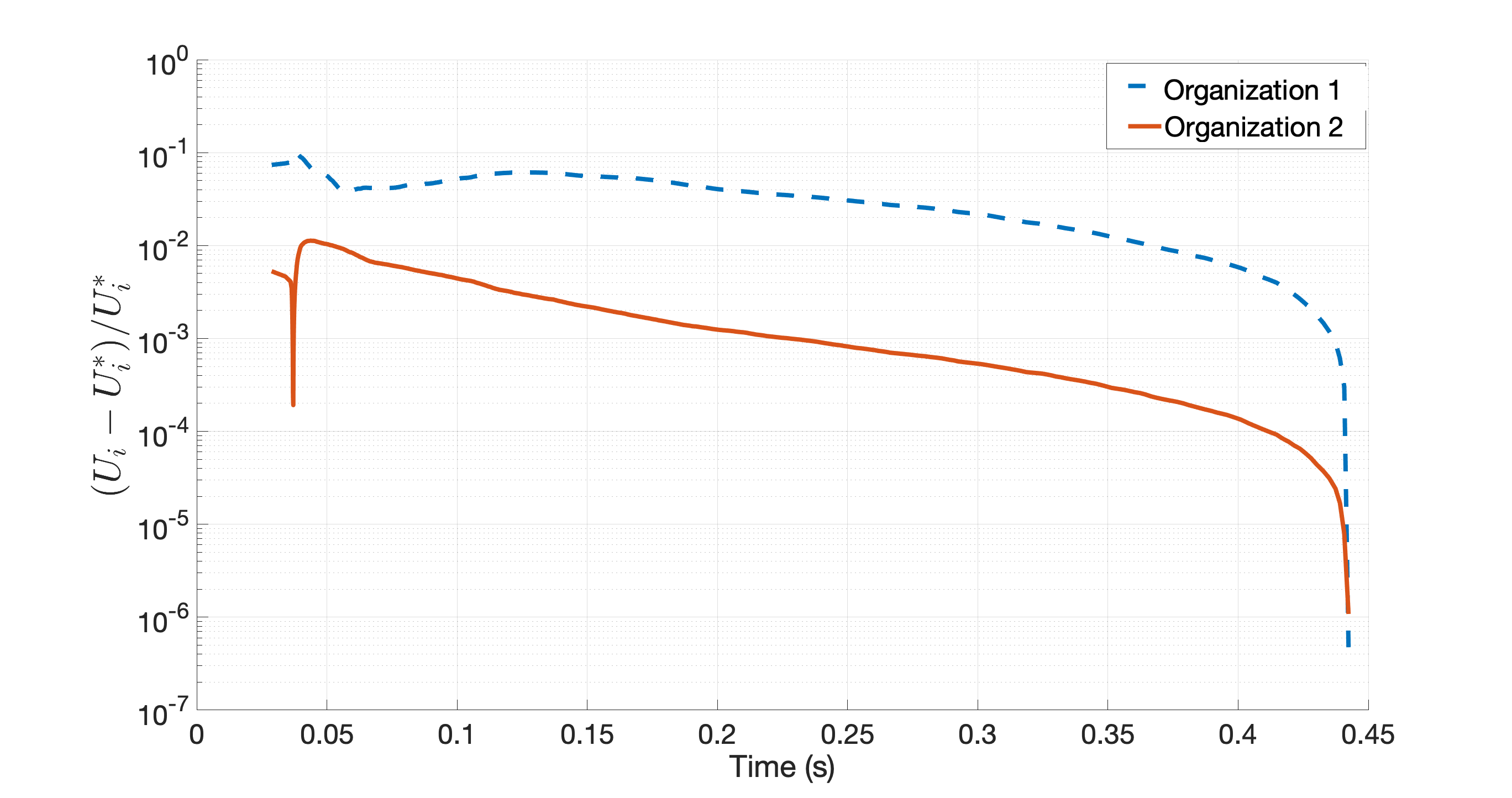}
    \caption{Relative sub-optimality error of Inexact-VR-SQVI versus time for Example 1. }
    \label{plotex1}
\end{figure}

 \begin{table}[htb]
	\centering
	\caption{The utility and CPU-time for exact and inexact-VR-SQVI for Example 1. }
\begin{tabular}{|c|c|c|}
\hline
  Methods      & Utility  & CPU-Time(s) \\ \hline
 Exact-VR-SQVI & {$U_1=0.7065e+4, \ U_2=4.0589e+4 $} & 97.64  \\ \hline
 Inexact-VR-SQVI  & $U_1=0.7065e+4, \ U_2=4.0589e+4$ & 0.50  \\ \hline
\end{tabular}
\label{Ex1Table}
\end{table}

 \textbf{Example 2.} This example includes the same network topology as in Example 1, that is, the one depicted in Figure \ref{f2}. The problem's parameters are selected as those in Example 1 except $P_{ij}(\cdot)$'s functions. 
In particular, in this example we consider 
\begin{align*}
   & P_{11}(Q)=50 \sqrt{10Q_{11} - Q_{21} - Q_{22} + 130}, \quad P_{12}(Q)=30 \sqrt{12Q_{12} - Q_{21} - 2Q_{22} + 135},\\
&    P_{21}(Q)=40 \sqrt{11Q_{21} - Q_{11} - Q_{12} + 123 }, \quad P_{22}(Q)=20 \sqrt{12Q_{22} - Q_{11} - Q_{12} + 135}.
\end{align*}

Therefore, the constraints related to \bmo{the} minimum requirement for blood collection in this example are: 
\begin{equation}
  \begin{aligned}\label{const_nonlinear}
  &50 \sqrt{10Q_{11} - Q_{21} - Q_{22} + 130} + 40 \sqrt{11Q_{21} - Q_{11} - Q_{12} + 123 } \geq 1200,\\
  &30 \sqrt{12Q_{12} - Q_{21} - 2Q_{22} + 135} + 20 \sqrt{12Q_{22} - Q_{11} - Q_{12} + 135} \geq 1100.
 \end{aligned}
 \end{equation}

Implementing Inexact-VR-SQVI on this example leads to the following solutions:
 \begin{align*}
    Q_{11}^* = 69.72,\quad Q_{12}^* = 40.00,\quad Q_{21}^* = 61.89,\quad Q_{22}^* = 70.00,
\end{align*}
and   $U_1(Q^*) = 3234401$ and  $U_2(Q^*) =264865 $.


It is worth noting that since the constraints set in this example is described by nonlinear functional constraints in \eqref{const_nonlinear}, the projection operator requires a more computational time as it is reflected in the running time of Exact-VR-SQVI--see Table \ref{Ex2table}. However, the Inexact-VR-SQVI has a lower per iteration complexity leading to  a far less computational time to achieve an accuracy of less than $10^{-6}$ as shown in Figure \ref{plotex2}. 

\begin{table}[htb]
\centering
    \caption{The utility and CPU-time for exact and inexact-VR-SQVI for Example 2. }
\begin{tabular}{|c|c|c|}
    \hline
     Methods      & Utility  & CPU-Time(s) \\ \hline
     Exact-VR-SQVI &   $U_1=3.234402e+6, \ U_2=2.648863e+6$ & 213  \\ \hline
     Inexact-VR-SQVI &  $U_1=3.234401e+6,\ U_2=2.648865e+6$  & 0.72  \\ \hline
    \end{tabular}
    \label{Ex2table}
\end{table}
\begin{figure}[htb]
    \centering
    \includegraphics[width = 9.5 cm]{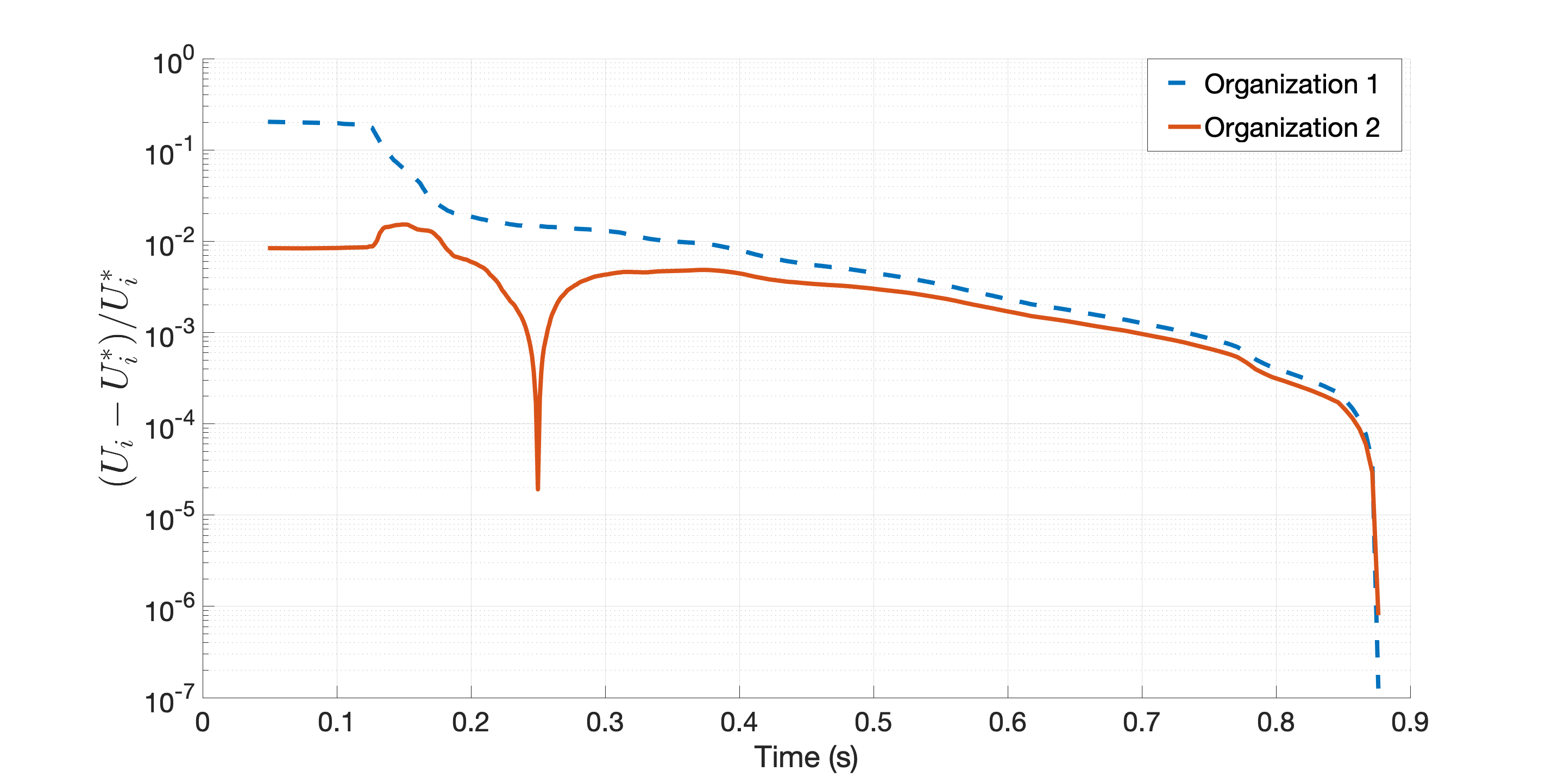}
    \caption{Relative sub-optimality error of Inexact-VR-SQVI versus time for Example 2. }
    \label{plotex2}
\end{figure}

\section{CONCLUDING REMARKS}\label{sec:conclude}
In this paper, we concentrate our efforts on the strongly-monotone stochastic quasi-variational inequality problems. An inexact variance reduced scheme 
{was developed}; moreover, the convergence rate and the oracle complexity of the proposed method are characterized. We believe that the proposed method represents  the first inexact scheme with a convergence guarantee to solve SQVI problems when the constraints are not easy to project. Additionally, we demonstrated the effectiveness and robustness of the proposed inexact method for solving blood donation problems in the numerical experiments. The results obtained in this paper are a crucial first step in examining more general cases. Future directions include investigating monotone and weakly-monotone SQVI problems with applications in various domains such as power systems and information security.

\bibliographystyle{siam}
\bibliography{biblio.bib}

\begin{thebibliography}{10}

\bibitem{webs1}
{\sc {American Red Cross}}, {\em Arizona blood services region. red cross
  issues emergency need for platelet donations to help cancer, surgery
  patients.}
\newblock
  \url{https://www.redcrossblood.org/local-homepage/articles.regionAbbreviation=001.type=article.p=1.html},
  2016.
\newblock Accessed February 1, 2022.

\bibitem{antipin2013second}
{\sc A.~S. Antipin, N.~Mijajlovi{\'c}, and M.~Ja{\'c}imovi{\'c}}, {\em A
  second-order iterative method for solving quasi-variational inequalities},
  Computational Mathematics and Mathematical Physics, 53 (2013), p.~258.

\bibitem{bertsekas2003convex}
{\sc D.~Bertsekas, A.~Nedi{\'c}, and A.~Ozdaglar}, {\em Convex analysis and
  optimization, ser}, vol.~1, Athena Scientific, 2003.

\bibitem{facchinei2014solving}
{\sc F.~Facchinei, C.~Kanzow, and S.~Sagratella}, {\em Solving
  quasi-variational inequalities via their kkt conditions}, Mathematical
  Programming, 144 (2014), pp.~369--412.

\bibitem{facchinei2007finite}
{\sc F.~Facchinei and J.-S. Pang}, {\em Finite-dimensional variational
  inequalities and complementarity problems}, Springer Science \& Business
  Media, 2007.

\bibitem{grant2014cvx}
{\sc M.~Grant and S.~Boyd}, {\em Cvx: Matlab software for disciplined convex
  programming, version 2.2}.
\newblock \url{http://cvxr.com/cvx/}, 2014.
\newblock Accessed March 15, 2022.

\bibitem{gurkan1996sample}
{\sc G.~Gurkan, A.~Y. Ozge, and S.~M. Robinson}, {\em Sample-path solution of
  stochastic variational inequalities, with applications to option pricing}, in
  Proceedings Winter Simulation Conference, J.~M. Charnes, D.~J. Morrice, D.~T.
  Brunner, and J.~J. Swain, eds., IEEE, 1996, pp.~337--344.

\bibitem{hamedani2018primal}
{\sc E.~Y. Hamedani and N.~S. Aybat}, {\em A primal-dual algorithm with line
  search for general convex-concave saddle point problems}, SIAM Journal on
  Optimization, 31 (2021), pp.~1299--1329.

\bibitem{he2015mirror}
{\sc N.~He, A.~Juditsky, and A.~Nemirovski}, {\em Mirror prox algorithm for
  multi-term composite minimization and semi-separable problems}, Computational
  Optimization and Applications, 61 (2015), pp.~275--319.

\bibitem{jalilzadeh2019proximal}
{\sc A.~Jalilzadeh and U.~V. Shanbhag}, {\em A proximal-point algorithm with
  variable sample-sizes ({PPAWSS}) for monotone stochastic variational
  inequality problems}, in 2019 Winter Simulation Conference (WSC),
  N.~Mustafee, K.-H.~G. Bae, S.~Lazarova-Molnar, M.~Rabe, C.~Szabo, P.~Haas,
  and Y.-J. Son, eds., IEEE, 2019, pp.~3551--3562.

\bibitem{malitsky2015projected}
{\sc Y.~Malitsky}, {\em Projected reflected gradient methods for monotone
  variational inequalities}, SIAM Journal on Optimization, 25 (2015),
  pp.~502--520.

\bibitem{malitsky2018proximal}
\leavevmode\vrule height 2pt depth -1.6pt width 23pt, {\em Proximal
  extrapolated gradient methods for variational inequalities}, Optimization
  Methods and Software, 33 (2018), pp.~140--164.

\bibitem{mijajlovic2015proximal}
{\sc N.~Mijajlovi{\'c} and M.~Jacimovi{\'c}}, {\em A proximal method for
  solving quasi-variational inequalities}, Computational Mathematics and
  Mathematical Physics, 55 (2015), p.~1981.

\bibitem{mijajlovic2019gradient}
{\sc N.~Mijajlovi{\'c}, M.~Ja{\'c}imovi{\'c}, and M.~A. Noor}, {\em
  Gradient-type projection methods for quasi-variational inequalities},
  Optimization Letters, 13 (2019), pp.~1885--1896.

\bibitem{nagurney2019competition}
{\sc A.~Nagurney and P.~Dutta}, {\em Competition for blood donations}, Omega,
  85 (2019), pp.~103--114.

\bibitem{nagurney2017supply}
{\sc A.~Nagurney, M.~Yu, and D.~Besik}, {\em Supply chain network capacity
  competition with outsourcing: a variational equilibrium framework}, Journal
  of Global Optimization, 69 (2017), pp.~231--254.

\bibitem{nesterov2006solving}
{\sc Y.~Nesterov and L.~Scrimali}, {\em Solving strongly monotone variational
  and quasi-variational inequalities}, Discrete and Continuous Dynamical
  Systems, 31 (2011), pp.~1383--1396.

\bibitem{noor2000new}
{\sc M.~A. Noor}, {\em New approximation schemes for general variational
  inequalities}, Journal of Mathematical Analysis and applications, 251 (2000),
  pp.~217--229.

\bibitem{noor2007existence}
\leavevmode\vrule height 2pt depth -1.6pt width 23pt, {\em Existence results
  for quasi variational inequalities}, Banach Journal of Mathematical Analysis,
  1 (2007), pp.~186--194.

\bibitem{noor1994general}
{\sc M.~A. Noor and W.~Oettli}, {\em On general nonlinear complementarity
  problems and quasi-equilibria}, Le Matematiche, 49 (1994), pp.~313--331.

\bibitem{pang2005quasi}
{\sc J.-S. Pang and M.~Fukushima}, {\em Quasi-variational inequalities,
  generalized nash equilibria, and multi-leader-follower games}, Computational
  Management Science, 2 (2005), pp.~21--56.

\bibitem{ravat2017existence}
{\sc U.~Ravat and U.~V. Shanbhag}, {\em On the existence of solutions to
  stochastic quasi-variational inequality and complementarity problems},
  Mathematical Programming, 165 (2017), pp.~291--330.

\bibitem{ryazantseva2007first}
{\sc I.~P. Ryazantseva}, {\em First-order methods for certain quasi-variational
  inequalities in a hilbert space}, Computational Mathematics and Mathematical
  Physics, 47 (2007), pp.~183--190.

\bibitem{salahuddin2004projection}
{\sc H.~Salahuddin}, {\em Projection methods for quasi-variational
  inequalities}, Mathematical and Computational Applications, 9 (2004),
  pp.~125--131.

\bibitem{stupia2015power}
{\sc I.~Stupia, L.~Sanguinetti, G.~Bacci, and L.~Vandendorpe}, {\em Power
  control in networks with heterogeneous users: A quasi-variational inequality
  approach}, IEEE Transactions on Signal Processing, 63 (2015), pp.~5691--5705.

\bibitem{tang2018securing}
{\sc X.~Tang, P.~Ren, and Z.~Han}, {\em Securing small cell networks under
  interference constraint: A quasi-variational inequality approach}, in 2018
  IEEE Global Communications Conference (GLOBECOM), IEEE, 2018, pp.~1--6.

\end{thebibliography}

\end{document}